\documentclass[12pt,reqno,a4paper]{amsart}
\usepackage[ansinew]{inputenc}
\usepackage{amsfonts}
\usepackage{latexsym}
\usepackage{mathrsfs}
\usepackage{amsmath}
\usepackage{amssymb}
\usepackage{eepic}
\usepackage{color}
\usepackage[margin=3.6cm]{geometry}
\usepackage{hyperref}
\hypersetup{
    colorlinks=true,
    linkcolor=blue,
    citecolor=red,
}

\newcommand{\veps}{\varepsilon}

\newcommand{\C}{\mathbb{C}}
\newcommand{\N}{\mathbb{N}}

\newtheorem{lettertheorem}{Theorem}

\newtheorem{letterlemma}[lettertheorem]{Lemma}

\newtheorem{defin}{Definition}[section]

\newtheorem{theorem}[defin]{Theorem}
\newtheorem{exa}{Example}
\newtheorem{lemma}[defin]{Lemma}

\newtheorem{rem}{Remark}
\newenvironment{remark}{\begin{rem}\rm}{\end{rem}}

\numberwithin{equation}{section}

\makeatletter
\renewcommand{\ps@myheadings}{%
\renewcommand{\@evenhead}%
{{\rm\thepage}\hfil{\sc H.~Yu, J.~Heittokangas, J.~Wang, and Z.~T.~Wen}\hfil}%
\renewcommand{\@oddhead}%
{\hfil{{\sc The $\varphi$-order for the Askey-Wilson divided differences}\hfil{\rm\thepage}}}%
\renewcommand{\@evenfoot}{}%
\renewcommand{\@oddfoot}{}%
}\makeatother 

\title{Meromorphic functions of finite $\varphi$-order and linear Askey-Wilson divided difference equations}

\author[Yu]{Hui Yu}

\author[Heittokangas]{Janne Heittokangas$^*$}
\address[Yu\\Heittokangas]{
Department of Physics and Mathematics, University of Eastern Finland, P.O.~Box 111, 80101 Joensuu, Finland}
\email{huiy@uef.fi\\janne.heittokangas@uef.fi}

\author[Wang]{Jun Wang}
\address[Wang]{School of Mathematical Sciences, Fudan University, Shanghai  200433, P.R.~China}
\email{majwang@fudan.edu.cn}

\author[Wen]{Zhi-Tao Wen}
\address[Wen]{Department of Mathematics, Shantou University, Shantou 515063, Guangdong, P.R.~China}
\email{zhtwen@stu.edu.cn}

\thanks{$^*$Corresponding author.}

\begin{document}
\maketitle

\begin{abstract}
The growth of meromorphic solutions of linear difference equations containing Askey-Wilson divided difference operators
is estimated. The $\varphi$-order is used as a general growth indicator, which covers the growth spectrum between
the logarithmic order $\rho_{\log}(f)$ and the classical order $\rho(f)$ of a meromorphic function $f$.

\medskip
\noindent
\textsc{Key words:} Askey-Wilson divided difference operator, Askey-Wilson divided difference equation, 
lemma on the logarithmic difference, meromorphic function, $\varphi$-order.

\medskip
\noindent
\textsc{MSC 2020:} Primary 39A13; Secondary 30D35.
\end{abstract}

\renewcommand{\thefootnote}{ }
\footnote{}

%%%%%%%%%%%%%%%%%%%%%%%%%%%%%%%%%%%%%%
% SECTION 1
%%%%%%%%%%%%%%%%%%%%%%%%%%%%%%%%%%%%%%

\section{Introduction}

Suppose that $q$ is a complex number satisfying $0<|q|<1$. In 1985, Askey and Wilson evaluated a $q$-beta integral 
\cite[Theorem~2.1]{AW}, which allowed them to construct a family of orthogonal polynomials \cite[Theorems~2.2--2.5]{AW}.
These polynomials are eigensolutions of a second order difference equation \cite[p.~36]{AW} that involves a divided difference operator $\mathcal{D}_q$ currently known as the \emph{Askey-Wilson operator}. We will define 
$\mathcal{D}_q$ below and call it the \emph{AW-operator} for brevity. In general, any three
consecutive orthogonal polynomials satisfy a certain three term recurrence relation, see \cite[p.~4]{AW} or
\cite[p.~42]{S}.

Recently,  Chiang and Feng \cite{CF} have obtained a full-fledged Nevanlinna theory for meromorphic functions of finite
logarithmic order with respect to the AW-operator on the complex plane $\C$. The concluding remarks in \cite{CF} admit
that the logarithmic order of growth appears to be restrictive, even though this class contains a large family of
important meromorphic functions. This encourages us to generalize some of the results in \cite{CF} in such a way that
the associated results for finite logarithmic order follow as special cases.

Let $\varphi:(R_0,\infty)\to (0,\infty)$ be a non-decreasing unbounded function.
The $\varphi$-order of a meromorphic function $f$ in $\mathbb{C}$ was introduced in \cite{HWWY} as the quantity
	$$
    \rho_{\varphi}(f)= \limsup_{r\to\infty}\frac{\log T(r,f)}{\log\varphi(r)}.
    $$
Prior to \cite{HWWY}, the $\varphi$-order was used as a growth indicator for meromorphic functions in the unit disc in \cite{CHR}.
In the plane case, the logarithmic order $\rho_{\log}(f)$ and the  classical order $\rho(f)$  of $f$ follow as special cases when choosing 
$\varphi(r)=\log r$ and $\varphi(r)=r$, respectively. This leads us to impose a global growth restriction 
	\begin{equation}\label{general-restriction}
	\log r\leq \varphi(r)\leq r,\quad r\geq R_0.
	\end{equation}
Here and from now on, the notation $r\geq R_0$ is being used to express that the associated inequality is valid ''for all
$r$ large enough''. 

For an entire function $f$, the Nevanlinna characteristic $T(r,f)$ can be replaced with the logarithmic maximum modulus $\log M(r,f)$ in
the quantities $\rho(f)$ and $\rho_{\log}(f)$ by using a well-known relation between $T(r,f)$ and $\log M(r,f)$, see \cite[p.~23]{Rubel}. 
The same is true for the $\varphi$-order, namely
	\begin{equation}\label{varphi-logM}
	\rho_\varphi(f)=\limsup_{r\to\infty}\frac{\log \log M(r,f)}{\log\varphi(r)},
	\end{equation}
provided that $\varphi$ is subadditive, that is, $\varphi(a+b)\leq \varphi(a)+\varphi(b)$ for all $a,b\geq R_0$.
In particular, this gives  $\varphi(2r)\leq 2\varphi(r)$, which yields \eqref{varphi-logM}.
Moreover, up to a normalization, subadditivity is implied by concavity, see \cite{HWWY} for details.

Following the notation in \cite{AW} (see \cite{CF} and \cite[p.~300]{Ismail} for an alternative notation), we suppose that $f(x)$ is a meromorphic function in $\C$, and let $x=\cos \theta$ and $z=e^{i\theta}$, where $\theta\in\C$.  Then, for $x\neq \pm 1$, the AW-operator is defined by
     \begin{equation}\label{(17)-CF}
     (\mathcal{D}_qf)(x):=\frac{\breve{f}(q^{\frac{1}{2}}e^{i\theta})-\breve{f}(q^{-\frac{1}{2}}e^{i\theta})}{\breve{e}(q^{\frac{1}{2}}e^{i\theta})-\breve{e}(q^{-\frac{1}{2}}e^{i\theta})}=\frac{\breve{f}(q^{\frac{1}{2}}e^{i\theta})-\breve{f}(q^{-\frac{1}{2}}e^{i\theta})}{(q^{\frac{1}{2}}-q^{-\frac{1}{2}})(z-1/z)/2},
     \end{equation}
where $ x=(z+1/z)/2=\cos\theta$, $z=e^{ i\theta}$, $e(x)=x$ and
    \begin{equation*}
    \breve{f}(z)=f((z + 1/z)/2)=f(x)=f(\cos \theta).
    \end{equation*}
In the exceptional cases $x=\pm 1$, we define 
	$$
	(\mathcal{D}_qf)(\pm 1)=\displaystyle\underset{x\neq \pm 1}{\lim_{x\to\pm 1}} (\mathcal{D}_qf)(x)
	=f'(\pm(q^{\frac{1}{2}}+q^{-\frac{1}{2}})/2).
	$$  
The branch of the square root in $z=x+\sqrt{x^2-1}$ can be fixed in such a way that for each $x\in\C$ there corresponds
a unique $z\in\C$, see  \cite{CF} and \cite[p.~300]{Ismail}. It is known that $\mathcal{D}_qf$ is meromorphic for a meromorphic function $f$ and entire for an entire function $f$ \cite[Theorem~2.1]{CF}. The AW-operator in \eqref{(17)-CF} can be written in the alternative form
     \begin{equation*}\label{df-another form}
     (\mathcal{D}_qf)(x)=\frac{f(\hat{x})-f(\check{x})}{\hat{x}-\check{x}},
     \end{equation*}
 where   $ x=(z+1/z)/2=\cos\theta$ and  
        \begin{equation*}\label{def-hat-x}
              \hat{x}=\frac{q^{\frac{1}{2}}z+q^{-\frac{1}{2}}z^{-1}}{2},\quad   \check{x}=\frac{q^{-\frac{1}{2}}z+q^{\frac{1}{2}}z^{-1}}{2}.
        \end{equation*}
Finally, AW-operators of arbitrary order are defined by $\mathcal{D}_q^{0}f=f$ and 
$\mathcal{D}_q^{n}f=\mathcal{D}_q(\mathcal{D}_q^{n-1}f)$, where $n\in\N$.

Lemma~\ref{Le-4.2-CF} below is a pointwise AW-type lemma on the logarithmic difference proved in \cite[Lemma~4.2]{CF},
and it is used in \cite{CF} to study the growth of meromorphic solutions of Askey-Wilson divided difference equations. We note that finite 
logarithmic order implies finite $\varphi$-order because of the growth restriction \eqref{general-restriction}.

\begin{letterlemma}\label{Le-4.2-CF}
Let $f(x)$ be a meromorphic function of finite  logarithmic order  such that $\mathcal{D}_q f\not\equiv 0$, and let $\alpha_1\in (0,1)$ be  arbitrary. Then there exists a constant $C_{\alpha_1}>0$ such that for $2(|q^{1/2}|+|q^{-1/2}|)|x|<R$,  we have      
\begin{equation}\label{(35)CF}
     \begin{split}
\log^+\left|\frac{\mathcal{D}_q f(x)}{f(x)}\right|
    &\leq 
    \frac{4R(|q^{1/2}-1|+|q^{-1/2}-1|)|x|}{(R-|x|)(R-2(|q^{1/2}|+|q^{-1/2}|)|x|)}\left(m(R,f)+m(R,1/f)\right)\\
    &\quad +2(|q^{1/2}-1|+|q^{-1/2}-1|)|x|\left(\frac{1}{R-|x|}+\frac{1}{R-2(|q^{1/2}|+|q^{-1/2}|)|x|}\right)\\
    &\quad \quad\times\left(n(R,f)+n(R,1/f)\right)\\
    &\quad+2C_{\alpha_1} (|q^{1/2}-1|^{\alpha_1}+|q^{-1/2}-1|^{\alpha_1}) |x|^{\alpha_1}\underset{|c_n|<R}{\sum}\frac{1}{|x-c_n|^{\alpha_1}}\\
    &\quad+2C_{\alpha_1} |q^{-1/2}-1|^{\alpha_1} |x|^{\alpha_1}\underset{|c_n|<R}{\sum}\frac{1}{|x+c(q)q^{-1/2}z^{-1}-q^{-1/2}c_n|^{\alpha_1}}\\
    &\quad+2C_{\alpha_1} |q^{1/2}-1|^{\alpha_1} |x|^{\alpha_1}\underset{|c_n|<R}{\sum}\frac{1}{|x-c(q)q^{1/2}z^{-1}-q^{1/2}c_n|^{\alpha_1}}+\log 2,
     \end{split}
     \end{equation}
where $c(q)=(q^{-1/2}-q^{1/2})/2  $ and $\{c_n\}$ is the combined sequence of zeros and poles of $f$.
\end{letterlemma}

The choice $R=r\log r$ in Lemma~\ref{Le-4.2-CF} is made in proving \cite[Theorem~3.1]{CF}, which is an AW-type lemma 
on the logarithmic difference asserting
    \begin{equation}\label{est-m}	
	m\left(r,\frac{\mathcal{D}_q f(x)}{f(x)}\right)
	=O\left((\log r)^{\rho_{\log}(f)-1+\varepsilon}\right),
    \end{equation}
where $\varepsilon>0$ is arbitrary and $f$ is a meromorphic function of finite logarithmic order $\rho_{\log}(f)$ 
such that $\mathcal{D}_qf\not\equiv 0$. The estimate \eqref{est-m} in turn is used to prove a growth estimate \cite[Theorem~12.4]{CF}
for meromorphic solutions of AW-divided difference equations, stated as follows.

\begin{lettertheorem}\label{CF12.4-or}
Let $a_0(x),a_1(x),\ldots,a_{n-1}(x)$ be entire functions such that 
    $$
    \rho_{\log}(a_0)>\max_{1\leq  j\leq n}\{\rho_{\log}(a_j)\}.
    $$
Suppose that  $f$ is an entire solution of the AW-divided difference equation
    $$
     \sum_{j=0}^na_j(x)\mathcal{D}_q^{j}f(x)=0,
    $$
where $a_n(x)=1$. Then $\rho_{\log}(f)\geq\rho_{\log}(a_0)+1$.
\end{lettertheorem}

Our main objectives are to find  $\varphi$-order analogues of the estimate \eqref{est-m} and of Theorem~\ref{CF12.4-or}.
A non-decreasing function $s:(R_0,\infty)\to(0,\infty)$ satisfying a global growth restriction
	\begin{equation}\label{assumption}
	r< s(r)\leq r^2,\quad r\geq R_0,
	\end{equation}
will take the role of $R$ in Lemma ~\ref{Le-4.2-CF}.
Suitable test functions for $\varphi$ and $s$ then are, for example,
	\begin{equation*}\label{test-functions}
	\varphi(r)=\log^\alpha r,\quad \varphi(r)=\exp(\log^\beta r),\quad \varphi(r)=r^\beta,
	\end{equation*}
along with $s(r)=r\log r$ and $s(r)=r^\alpha$, where $\alpha\in(1,2]$ and $\beta\in (0,1]$.

This paper is organized as follows. A generalization of Theorem~\ref{CF12.4-or} for meromorphic solutions in terms of the $\varphi$-order 
is given in Section~\ref{main result}. Two AW-type lemmas on the logarithmic difference in terms of the $\varphi$-order are given in 
Section~\ref{lemmas}. One of them will be among the most important individual tools later on. Section~\ref{N--} consists of lemmas on 
AW-type counting functions as well as on the Nevanlinna characteristic of $\mathcal{D}_q f$.  These lemmas
are crucial in proving the main results, which are Theorem~\ref{AW-th12.4} and \ref{AW-th12.4-non} below. The details of the proofs are given in Section~\ref{proofs}.

%%%%%%%%%%%%%%%%%%%%%%%%%%%%%%%%%%%%%%
% SECTION 2
%%%%%%%%%%%%%%%%%%%%%%%%%%%%%%%%%%%%%%
 
\section{Results on Askey-Wilson divided difference equations}\label{AW-section}\label{main result}

We consider the growth of meromorphic solutions of AW-divided difference equations
 \begin{equation}\label{AW-q-diff}
    \sum_{j=0}^na_j(x)\mathcal{D}_q^{j}f(x)=0
    \end{equation}
    and of the corresponding non-homogeneous AW-divided difference equations
    \begin{equation}\label{AW-q-diff-non}
    \sum_{j=0}^na_j(x)\mathcal{D}_q^{j}f(x)=a_{n+1}(x),
    \end{equation}
where $a_0,\ldots,a_{n+1}$ are meromorphic functions, and $a_0a_n\not\equiv 0$. 
The results that follow depend on growth parameters introduced in \cite{HWWY} and defined by
		\begin{equation}\label{liminf}
	\alpha_{\varphi,s}=\liminf_{r\to\infty}\frac{\log \varphi(r)}{\log \varphi(s(r))} \quad\text{and}\quad
	\gamma_{\varphi,s}=\liminf_{r\to\infty}\frac{\log\log \frac{s(r)}{r}}{\log \varphi(r)}.
	\end{equation}
Due to the assumptions \eqref{general-restriction} and \eqref{assumption},
we always have $ \alpha_{\varphi,s}\in[0,1]$ and $ \gamma_{\varphi,s}\in[-\infty,1]$.
From now on, we make a global assumption
    \begin{equation*}\label{global-s/r}
    \liminf_{r\to\infty}\frac{s(r)}{r}>1,
    \end{equation*}
which ensures that $ \gamma_{\varphi,s}\in[0,1]$. Further properties and relations related to the growth parameters $\alpha_{\varphi,s}$ and $\gamma_{\varphi,s}$ can be found in \cite{HWWY}.

Theorem~\ref{AW-th12.4} below reduces to Theorem~\ref{CF12.4-or} when choosing $\varphi(r)=\log r$ and $s(r)=r^2$ and when
the coefficients and solutions are entire functions.  
 
\begin{theorem}\label{AW-th12.4}
Suppose that $\varphi(r)$ is subadditive, and let $\alpha_{\varphi,s}$ and $\gamma_{\varphi,s}$ be the constants in 
\eqref{liminf}. Let $a_0,\ldots,a_{n}$ be meromorphic functions of finite $\varphi$-order such that 
    $$
    \rho_{\varphi}(a_0)>\max_{1\leq  j\leq n}\{\rho_\varphi(a_j)\}.
    $$
	\begin{itemize}
    \item[\textnormal{(a)}] 
    Suppose that $\displaystyle\limsup_{r\to\infty}\frac{s(r)}{r}=\infty$ and  that $s(r)$ is convex and differentiable.
	If $f$ is a non-constant meromorphic solution of \eqref{AW-q-diff}, then
    \begin{equation}\label{a-1-2.1}
    \rho_\varphi (f)  \geq   \alpha_{\varphi,s}^n \rho_\varphi(a_0).       
   	\end{equation}
   	Moreover, if the coefficients $a_0,\ldots, a_n$ are entire, then
    \begin{equation}\label{a-2-2.1}
    \rho_\varphi (f) \geq \alpha_{\varphi,s}^n \rho_\varphi(a_0)+\alpha_{\varphi,s}^n \gamma_{\varphi,s}.
    \end{equation}
	\item[\textnormal{(b)}]  
	Suppose that $\displaystyle\limsup_{r\to\infty}\frac{s(r)}{r}<\infty$. If $f$ is a non-constant meromorphic solution 
	of \eqref{AW-q-diff}, then $\rho_\varphi(f)\geq\alpha_{\varphi,s}^{n-1}\rho_\varphi(a_0)$.
\end{itemize}  
\end{theorem}

\begin{remark}
For  certain $\varphi(r)$,  for example,  for $\varphi(r)=\log^\alpha r$, where $\alpha\in(1,2]$, the conclusion of Theorem \ref{AW-th12.4}(a) is stronger than that of Theorem \ref{AW-th12.4}(b) due to  different choices of $s(r)$.   If the coefficients $a_0,\ldots, a_n$  are entire, then it follows from \eqref{liminf} and \eqref{a-2-2.1} that $\rho_\varphi (f) \geq \rho_\varphi(a_0)+1/\alpha$ in  Theorem \ref{AW-th12.4}(a) when choosing $s(r)=r^2$, which is stronger than the conclusion $\rho_\varphi(f)\geq\rho_\varphi(a_0)$ in Theorem \ref{AW-th12.4}(b) when choosing $s(r)=2r$.

On the other hand, the opposite is true for some suitable $\varphi(r)$. For instance,  choose $ \varphi(r)=r^\beta$, where $\beta\in (0,1]$, along with $s(r)=2r$ and $s(r)=r^2$, respectively. Then  we get $\rho_{\varphi}(f)\geq  \rho_{\varphi}(a_0)$ from Theorem \ref{AW-th12.4}(b), which is stronger than the conclusion $\rho_{\varphi}(f)\geq (1/2)^n\rho_{\varphi}(a_0)$  in Theorem \ref{AW-th12.4}(a), which in turn follows from \eqref{liminf} and  \eqref{a-1-2.1}.
\end{remark}

 The following result is a  growth estimate for meromorphic solutions of the non-homogeneous equations \eqref{AW-q-diff-non}.
\begin{theorem}\label{AW-th12.4-non}
Suppose that $\varphi(r)$ is subadditive. Let $a_0,\ldots,a_{n}$ be meromorphic functions of finite $\varphi$-order such that 
    $$
    \rho_{\varphi}(a_0)>\max_{1\leq  j\leq n+1}\{\rho_\varphi(a_j)\}.
    $$
   	If $f$ is a non-constant meromorphic solution 
	of \eqref{AW-q-diff-non}, then $\rho_\varphi(f)\geq \alpha_{\varphi,s}^{n-1} \rho_\varphi(a_0)$.
\end{theorem}

The proofs of Theorems~\ref{AW-th12.4} and \ref{AW-th12.4-non} in Section~\ref{proofs} are based on an AW-type lemma on the
logarithmic difference discussed in Section~\ref{lemmas} as well as on estimates for AW-type counting functions discussed in Section~\ref{N--}.

%%%%%%%%%%%%%%%%%%%%%%%%%%%%%%%%%%%%%%
% SECTION 3
%%%%%%%%%%%%%%%%%%%%%%%%%%%%%%%%%%%%%%

\section{Estimates for the Askey-Wilson type\\ logarithmic  difference }\label{lemmas}

Lemma~\ref{m-AW} below is an AW-type lemma on the logarithmic difference, which reduces to \cite[Theorem 3.1]{CF} 
when choosing $\varphi(r)=\log r$ and $s(r)=r^2$. The proof uses the notation  $g(r)\lesssim h(r)$ to express 
that  there exists a constant $C\geq 1$ such that $g(r)\leq Ch(r)$ for all $r\geq R_0$. 

\begin{lemma}\label{m-AW}
Let $f$ be a meromorphic function of finite $\varphi$-order $\rho_{\varphi}(f)$ such that $\mathcal{D}_q f\not\equiv 0$. Let
$\alpha_{\varphi,s}$ and $\gamma_{\varphi,s}$ be the constants in \eqref{liminf}, let  $\varepsilon>0$,
and denote $|x|=r$. 
\begin{itemize}
\item[\textnormal{(a)}] If $\displaystyle\limsup_{r\to\infty}\frac{s(r)}{r}=\infty$ and if $s(r)$
is convex and differentiable, then
	\begin{equation*}	
	m\left(r,\frac{\mathcal{D}_q f(x)}{f(x)}\right)
	=O\left(\frac{\varphi(s(r))^{\rho_\varphi(f)+\frac{\varepsilon}{2}}}{\log\frac{s(r)}{r}}+1\right)
	=O\left({\varphi(s(r))^{\rho_\varphi(f)-\alpha_{\varphi,s}\gamma_{\varphi,s}+\varepsilon}}\right).
    \end{equation*}
\item[\textnormal{(b)}] If $\displaystyle\limsup_{r\to\infty}\frac{s(r)}{r}<\infty$ and if
$\varphi(r)$ is subadditive, then
    \begin{equation*}
	m\left(r,\frac{\mathcal{D}_q f(x)}{f(x)}\right)=O\left(\varphi(r)^{\rho_\varphi(f)+\varepsilon}\right).
    \end{equation*}
\end{itemize}
\end{lemma}
 
\begin{proof}
(a) By the proof of  \cite[Lemma~3.1(a)]{HWWY}, there exist non-decreasing functions 
$u,v:[1,\infty)\to(0,\infty)$ with the following properties:
\begin{itemize}
\item[(1)] $r<u(r)<s(r)$ and $r<v(r)<s(r)$ for all $r\geq R_0$,
\item[(2)] $u(r)/r\to\infty$ and $v(r)/r\to\infty$ as $r\to\infty$,
\item[(3)] $2^{-1}s(r)\leq v(u(r))\leq s(r)$ for all $r\geq R_0$,
\item[(4)] $2\log (u(r)/r)\leq \log (s(r)/r)\leq 2u(r)/r$ for all $r\geq R_0$.
\end{itemize} 
Using the standard estimate
    \begin{equation*}\label{n-esti-v(r)}
	N(v(r),f)-N(r,f)=\int_r^{v(r)}\frac{n(t,f)}{t}\, dt
	\geq n(r,f)\log\frac{v(r)}{r}
	\end{equation*}
and  the properties (3) and (4), we deduce that
    \begin{equation}\label{n-v-(a)}
    n(u(r),f)	\leq \frac{T(s(r),f)}{\log\frac{s(r)}{2r}-\log\frac{u(r)}{r}}\lesssim \frac{T(s(r),f)}{\log\frac{s(r)}{r}}\lesssim\frac{\varphi(s(r))^{\rho_\varphi(f)+\frac{\varepsilon}{2}}}{\log\frac{s(r)}{r}},
	\end{equation}
and similarly for $n(u(r),1/f)$. Choose $R=u(r)$. We integrate \eqref{(35)CF} from $0$ to $2\pi$, and we make use of 
the properties (1) and (4) together with \eqref{n-v-(a)} and formulas (63)--(64) in \cite{CF}, and obtain
    \begin{equation*}\label{dqf}
    \begin{split}
    m\left(r,\frac{\mathcal{D}_q f(x)}{f(x)}\right)&\lesssim\frac{T(u(r),f)}{u(r)/r}+ {n(u(r),f)+n(u(r),1/f)} +1\\
    &\lesssim\frac{\varphi(s(r))^{\rho_\varphi(f)+\frac{\varepsilon}{2}}}{\log\frac{s(r)}{r}}+1.
    \end{split}
    \end{equation*} 
This proves the first identity in Case (a).

From \eqref{liminf}, we get
      $$
      \alpha_{\varphi,s}\gamma_{\varphi,s}\leq
      \liminf_{r\to\infty}\left(\frac{\log \varphi(r)}{\log\varphi(s(r))}\cdot\frac{\log\log\frac{s(r)}{r}}{\log \varphi(r)}\right)
      =\liminf_{r\to\infty}\frac{\log\log\frac{s(r)}{r}}{\log\varphi(s(r))},
      $$
and so
     $$
     \log\frac{s(r)}{r}\geq\varphi(s(r))^{\alpha_{\varphi,s}\gamma_{\varphi,s}-\frac{\varepsilon}{2}},\quad r\geq R_0.
     $$
Recall from \cite[Corollary~4.3]{HWWY} that, for a non-constant meromorphic function $f$ of finite $\varphi$-order $\rho_\varphi(f)$, we have $\rho_\varphi(f)\geq \alpha_{\varphi,s} \gamma_{\varphi,s}$. Thus
      \begin{equation}\label{unbounded-var}
      \frac{\varphi(s(r))^{\rho_\varphi(f)+\frac{\varepsilon}{2}}}{\log\frac{s(r)}{r}}\leq 
      \varphi(s(r))^{\rho_\varphi(f)-\alpha_{\varphi,s}\gamma_{\varphi,s}+\varepsilon},\quad r\geq R_0,
      \end{equation}
where the right-hand side tends to infinity as $r\to\infty$.
This proves the second identity in Case (a).

(b) By the assumptions on $s(r)$, there exists a constant $C\in (1,\infty)$ such that
$r<s(r)<Cr$ for all $r\geq R_0$. We choose $R=Br$, where 
     \begin{equation}\label{B-defi}
     B=\max\{[C],[ 2(|q^{1/2}|+|q^{-1/2}|)]\}+1
     \end{equation}
is an integer.  
Integrating \eqref{(35)CF} from $0$ to $2\pi$ and making use of  formulas (63)--(64) in \cite{CF} together with
    \begin{equation}\label{2br}
	T(2r,f)\geq\int_r^{2r}\frac{n(t,f)}{t}\, dt
	\geq n(r,f)\log 2,
	\end{equation}
we obtain 
    \begin{equation}\label{m-esti-var(2Br)}
    \begin{split}
    m\left(r,\frac{\mathcal{D}_q f(x)}{f(x)}\right) &\lesssim  T(Br,f)+n(Br,f)+n(Br,1/f)+1\\
    &\lesssim\varphi(2Br)^{\rho_\varphi(f)+\varepsilon}+1.
    \end{split}
    \end{equation}
Since the subadditivity of $\varphi$  yields $\varphi(2Br)\leq 2B\varphi(r)$, the assertion  follows from \eqref{m-esti-var(2Br)}.  
This completes the proof.
\end{proof}

Lemma~\ref{log+} below is a pointwise estimate for the AW-type logarithmic difference that holds outside of an exceptional set.
The result reduces to \cite[Theorem 3.2]{CF} when choosing $\varphi(r)=\log r$ and $s(r)=r^2$.

\begin{lemma}\label{log+}
Let $f$ be a meromorphic function of finite $\varphi$-order $\rho_{\varphi}(f)$ such that $\mathcal{D}_q f\not\equiv 0$. 
Let $\alpha_{\varphi,s}>0$ and $\gamma_{\varphi,s}$ be the constants in \eqref{liminf}, let $\varepsilon>0$,
and denote $|x|=r$.  Suppose that $\varphi(r)$ is continuous and satisfies
    \begin{equation}\label{log varphi/log r=0}
    \displaystyle\limsup_{r\to\infty}\frac{\log \varphi(r)}{\log r}=0.
    \end{equation}
 \begin{itemize}
\item[\textnormal{(a)}] If $\displaystyle\limsup_{r\to\infty}\frac{s(r)}{r}=\infty$ and if $s(r)$
is convex and differentiable, then
	\begin{equation*}	
	\log^+\left|\frac{\mathcal{D}_q f(x)}{f(x)}\right|
	=O\left(\frac{\varphi(s(r))^{\rho_\varphi(f)+\frac{\varepsilon}{2}}}{\log\frac{s(r)}{r}}+1\right)
	=O\left({\varphi(s(r))^{\rho_\varphi(f)-\alpha_{\varphi,s}\gamma_{\varphi,s}+\varepsilon}}\right)
    \end{equation*}
    holds outside of an exceptional set of finite logarithmic measure.
\item[\textnormal{(b)}] If $\displaystyle\limsup_{r\to\infty}\frac{s(r)}{r}<\infty$ and if
$\varphi(r)$ is subadditive, then
    \begin{equation*}
	\log^+\left|\frac{\mathcal{D}_q f(x)}{f(x)}\right|=O\left(\varphi(r)^{\rho_\varphi(f)+\varepsilon}\right)
    \end{equation*}
       holds outside of an exceptional set of finite logarithmic measure.
\end{itemize}
\end{lemma}

\begin{proof} 
We modify the proof of \cite[Theorem 3.2]{CF} as follows.

(a) Denote
    \begin{equation}\label{dn-def}
    \{d_n\}:=\{c_n\}\cup\{q^{1/2}c_n\}\cup\{q^{-1/2}c_n\},
    \end{equation}
where $\{c_n\}$ is the combined sequence of zeros and poles of $f$. Let
    $$
    E_n=\left\{r:r\in\left[|d_n|-\frac{|d_n|}{\varphi(|d_n|+3)^{\frac{\rho_\varphi(f)+\varepsilon}{\alpha_{\varphi,s}}}},
    \,|d_n|+\frac{|d_n|}{\varphi(|d_n|+3)^{\frac{\rho_\varphi(f)+\varepsilon}{\alpha_{\varphi,s}}}}\right]\right\}
    $$
and $E=\cup_n E_n$, where $ \alpha_{\varphi,s}\in(0,1]$ is defined in \eqref{liminf}.
In what follows, we consider $r\not\in E$.  We proceed to prove that
    \begin{equation}\label{|x-d_n|}
     |x-d_n|\geq \frac{|x|}{2\varphi(|x|+3)^{\frac{\rho_\varphi(f)+\varepsilon}{\alpha_{\varphi,s}}}},\quad |x|=r\geq R_0.
    \end{equation}
The proof is divided into three cases in each of which $|x|\geq R_0$. 
\begin{itemize}
\item[(1)] Suppose that $|x|<|d_n|-\frac{|d_n|}{\varphi(|d_n|+3)^{\frac{\rho_\varphi(f)+\varepsilon}{\alpha_{\varphi,s}}}}$. 
From \eqref{log varphi/log r=0}, the function $\frac{|x|}{\varphi(|x|+3)^{\frac{\rho_\varphi(f)+\varepsilon}{\alpha_{\varphi,s}}}}$ is increasing, and so
    \begin{eqnarray*}
    |x-d_n| &\geq& ||x|-|d_n||\geq \frac{|d_n|}{\varphi(|d_n|+3)^{\frac{\rho_\varphi(f)+\varepsilon}{\alpha_{\varphi,s}}}}
    \geq \frac{|x|}{2\varphi(|x|+3)^{\frac{\rho_\varphi(f)+\varepsilon}{\alpha_{\varphi,s}}}}.
    \end{eqnarray*}
    \item[(2)] Suppose that $|d_n|+\frac{|d_n|}{\varphi(|d_n|+3)^{\frac{\rho_\varphi(f)+\varepsilon}{\alpha_{\varphi,s}}}}
    \leq |x|-\frac{|x|}{\varphi(|x|+3)^{\frac{\rho_\varphi(f)+\varepsilon}{\alpha_{\varphi,s}}}}$. Clearly, 
    \begin{eqnarray*}
    |x-d_n| &\geq& \frac{|d_n|}{\varphi(|d_n|+3)^{\frac{\rho_\varphi(f)+\varepsilon}{\alpha_{\varphi,s}}}}
    +\frac{|x|}{\varphi(|x|+3)^{\frac{\rho_\varphi(f)+\varepsilon}{\alpha_{\varphi,s}}}}
    \geq \frac{|x|}{2\varphi(|x|+3)^{\frac{\rho_\varphi(f)+\varepsilon}{\alpha_{\varphi,s}}}}.
    \end{eqnarray*}
    \item[(3)]  Suppose that $|d_n|+\frac{|d_n|}{\varphi(|d_n|+3)^{\frac{\rho_\varphi(f)+\varepsilon}{\alpha_{\varphi,s}}}}< |x|$ and 
    $$
    |x|-\frac{|x|}{\varphi(|x|+3)^{\frac{\rho_\varphi(f)+\varepsilon}{\alpha_{\varphi,s}}}}
    \leq |d_n|+\frac{|d_n|}{\varphi(|d_n|+3)^{\frac{\rho_\varphi(f)+\varepsilon}{\alpha_{\varphi,s}}}}.
    $$ 
    Then we have $|x-d_n|\geq \frac{|d_n|}{\varphi(|d_n|+3)^{\frac{\rho_\varphi(f)+\varepsilon}{\alpha_{\varphi,s}}}}$ 
    and $|x|=|d_n|(1+o(1))$ as $|x|\to\infty$ (or as $n\to\infty$). This yields \eqref{|x-d_n|} by the 
    continuity~of~$\varphi(r)$.
\end{itemize}
Keeping in mind that $r\not\in E$, this completes the proof of \eqref{|x-d_n|}. 

Let $\alpha_1\in (0,1)$. From \eqref{|x-d_n|},
    \begin{equation}\label{52}
    \sum_{|c_n|<R}\frac{1}{|x-c_n|^{\alpha_1}}\leq
    \frac{2^{\alpha_1}\varphi(|x|+3)^{\frac{\alpha_1(\rho_\varphi(f)+\varepsilon)}{\alpha_{\varphi,s}}}}{|x|^{\alpha_1}}
    \left(n(R,f)+n(R,1/f)\right).
    \end{equation}
From \eqref{log varphi/log r=0}--\eqref{|x-d_n|}, we have, for all $|x|$ sufficiently large and hence for all $|z|$ sufficiently large,
    \begin{equation*}
    \begin{split}
    |x+c(q)q^{-1/2}z^{-1}-q^{-1/2}c_n|
    &\geq |x-q^{-1/2}c_n|-|c(q)q^{-1/2}z^{-1}|     \geq \frac{|x|}{3\varphi(|x|+3)^{\frac{\rho_\varphi(f)+\varepsilon}{\alpha_{\varphi,s}}}},
    \end{split}
    \end{equation*}
    and similarly for
    $
    |x-c(q)q^{1/2}z^{-1}-q^{1/2}c_n|,
    $
    where $c(q)=(q^{-1/2}-q^{1/2})/2  $.
Therefore, 
    \begin{eqnarray}
    &&\underset{|c_n|<R}{\sum}\frac{1}{|x+c(q)q^{-1/2}z^{-1}-q^{-1/2}c_n|^{\alpha_1}}+ 
    \underset{|c_n|<R}{\sum}\frac{1}{|x-c(q)q^{1/2}z^{-1}-q^{1/2}c_n|^{\alpha_1}}\nonumber\\
    &&\qquad\leq \frac{2\cdot3^{\alpha_1}\varphi(|x|+3)^{\frac{\alpha_1(\rho_\varphi(f)+\varepsilon)}{\alpha_{\varphi,s}}}}{|x|^{\alpha_1}}
    \left(n(R,f)+n(R,1/f)\right). \label{55}
    \end{eqnarray}

We make use of the proof of Lemma \ref{m-AW}, according to which there exist non-decreasing functions 
$u,v:[1,\infty)\to(0,\infty)$ satisfying the aforementioned properties (1)--(4). 
Choose $R=u(r)$ and $\alpha_1=\frac{\alpha_{\varphi,s}\varepsilon}{4(\rho_\varphi(f)+\varepsilon)}\in (0,1)$. 
Since $\varepsilon>0$ is arbitrary, it follows from \eqref{n-v-(a)} that
    \begin{equation}\label{n-u-1/4}
    n(u(r),f)	 \lesssim\frac{\varphi(s(r))^{\rho_\varphi(f)+\frac{\varepsilon}{4}}}{\log\frac{s(r)}{r}}.
    \end{equation}
By substituting  \eqref{52}--\eqref{n-u-1/4} into \eqref{(35)CF},  and  by using \eqref{unbounded-var}, we have
    \begin{equation}\label{(57)CF}
    \begin{split}
       \log^+\left|\frac{\mathcal{D}_q f(x)}{f(x)}\right|
       &\lesssim \frac{T(u(r),f))}{u(r)/r}+\frac{n(u(r),f)+n(u(r),1/f)}{u(r)/r}\\
       &\quad+ \varphi(r+3)^{\frac{\alpha_1}{{\alpha_{\varphi,s}}}(\rho_\varphi(f)+\varepsilon)}\cdot \frac{\varphi(s(r))^{\rho_\varphi(f)+\frac{\varepsilon}{4}}}{\log\frac{s(r)}{r}}+1\\
      & \lesssim \frac{\varphi(s(r))^{\rho_\varphi(f)+\frac{\varepsilon}{2}}}{\log\frac{s(r)}{r}}+1\lesssim  {\varphi(s(r))^{\rho_\varphi(f)-\alpha_{\varphi,s}\gamma_{\varphi,s}+\varepsilon}},\quad r\not\in E.
    \end{split}
     \end{equation} 
     
By \eqref{(57)CF}, it suffices to prove that the logarithmic measure of the exceptional set $E$ is finite.
We recall from \cite[p.~249]{BIY} that, for a meromorphic function $h(x)$,  
    \begin{equation*}\label{equal-log-order-N}
	n(r,h(cx))=n(|c|r,h(x)),\quad c\in\C\setminus\{0\}.
	\end{equation*} 
We apply this formula to the functions  $f(q^{-1/2}x)$ and $f(q^{1/2}x)$ and make use of \cite[Lemmas~4.1--4.2]{HWWY} to get
     $$
    \lambda_{\varphi}=\rho_\varphi(n(Ar,f)+n(Ar,1/f))\leq \frac{\rho_{\varphi}(f)}{\alpha_{\varphi,s}}<\infty,
     $$
where $\lambda_{\varphi} $ is  the $\varphi$-exponent of convergence of the sequence $\{d_n\}$ defined in \eqref{dn-def}, 
and $A=\max\{1,|q|^{-1/2},|q|^{1/2}\}$. For $N\geq R_0$ and a given sufficiently small $\delta>0$, we have 
$\frac{1}{\varphi(|d_N|)^{\frac{\rho_\varphi(f)+\varepsilon}{\alpha_{\varphi,s}}}}<\delta$. Using the fact that $\log(1 + |x|)\leq |x|$ for all $|x|\geq 0$, the constant $C_\delta=\frac{2}{1-\delta}>0$ satisfies
     \begin{equation*}\label{log-ine-C_delta}
     \log\frac{1+\frac{1}{\varphi(|d_N|)^{\frac{\rho_\varphi(f)+\varepsilon}{\alpha_{\varphi,s}}}}}{1-\frac{1}{\varphi(|d_N|)^{\frac{\rho_\varphi(f)+\varepsilon}{\alpha_{\varphi,s}}}}}
     \leq C_\delta \cdot\frac{1}{\varphi(|d_N|)^{\frac{\rho_\varphi(f)+\varepsilon}{\alpha_{\varphi,s}}}},\quad N\geq R_0.
     \end{equation*}
Therefore,
     \begin{equation*}
     \begin{split}
     \text{log-meas}\,  (E)
     &=\left(\int_{E\cap[1,|d_N|]}+\int_{E\cap[|d_N|,\infty)}\right)\,\frac{dt}{t}\\
     &\leq \log |d_N|+\sum_{n=N}^\infty\int_{E_n}\,\frac{dt}{t}
     = \log |d_N|+\sum_{n=N}^\infty\log\frac{1+\frac{1}{\varphi(|d_n|)^{\frac{\rho_\varphi(f)+\varepsilon}{\alpha_{\varphi,s}}}}}{1-\frac{1}{\varphi(|d_n|)^{\frac{\rho_\varphi(f)+\varepsilon}{\alpha_{\varphi,s}}}}}\\
     &\leq \log |d_N|+C_\delta\sum_{n=N}^\infty\frac{1}{\varphi(|d_n|)^{\lambda_{\varphi}+\frac{\varepsilon}{\alpha_{\varphi,s}}}}<\infty,
     \end{split}
     \end{equation*}
which yields the assertion.

(b)\, By making use of the proof of Lemma~\ref{m-AW}(b) and following the
same method as in Case (a) above, we obtain \eqref{52} and \eqref{55}. Choose $R=Br$ and $\alpha_1=\frac{\alpha_{\varphi,s}\varepsilon}{2(\rho_\varphi(f)+\varepsilon)}\in (0,1)$, where $B$ is defined in \eqref{B-defi}. Then by substituting \eqref{2br}, \eqref{52} and \eqref{55} into \eqref{(35)CF}, we have
        \begin{equation*}
    \begin{split}
       \log^+\left|\frac{\mathcal{D}_q f(x)}{f(x)}\right|
       &\lesssim   T(Br,f)+n(Br,f)+n(Br,1/f)\\
       &\quad+ \varphi(r+3)^{\frac{\alpha_1}{{\alpha_{\varphi,s}}}(\rho_\varphi(f)+\varepsilon)}\cdot \varphi(2Br)^{\rho_{\varphi}(f)+\frac{\varepsilon}{2}}+1\\
      &\leq {\varphi(2Br)^{\rho_\varphi(f)+\varepsilon}},\quad r\not\in E.
    \end{split}
     \end{equation*} 
Then the assertion follows from the subadditivity of $\varphi$, that is, $\varphi(2Br)\leq 2B\varphi(r)$.  Similarly as in Case (a) above, we deduce that the logarithmic measure of the exceptional set $E$ is finite.  This completes the proof.
\end{proof}

%%%%%%%%%%%%%%%%%%%%%%%%%%%%%%%%%%%%%%
% SECTION 4
%%%%%%%%%%%%%%%%%%%%%%%%%%%%%%%%%%%%%%

\section{Askey-Wilson type counting functions\\ and  characteristic functions}\label{N--}

In this section we state three lemmas, whose proofs are just minor modifications of the corresponding 
results in \cite{CF}. For a non-constant meromorphic function $f$, it follows from \cite[Lemmas~4.1--4.2]{HWWY} that $\rho_\varphi(f)\geq  \alpha_{\varphi,s}\lambda_\varphi+\alpha_{\varphi,s} \gamma_{\varphi,s}$ and, if $\alpha_{\varphi,s}>0$, then
    \begin{equation*}\label{n-5.1}
        n(r,a,f)=O(\varphi(r)^{\lambda_\varphi+\varepsilon})\leq O\left(\varphi(r)^{\frac{\rho_\varphi(f)}{\alpha_{\varphi,s}}-\gamma_{\varphi,s}+\varepsilon}\right),
        \end{equation*}
where $\lambda_\varphi$ is the $\varphi$-exponent of convergence of the $a$-points of $f$.  

Lemma \ref{Th5.1-CF} below is essential in  proving Lemma~\ref{N-AW}, and it reduces to \cite[Theorem~5.1]{CF} when choosing  $\varphi(r)=\log r$ and $s(r)=r^2$.
   
\begin{lemma}\label{Th5.1-CF}
Let $f$ be a non-constant meromorphic function of finite $\varphi$-order $\rho_{\varphi}(f)$. Suppose that $\varphi(r)$ is subadditive. Let $\alpha_{\varphi,s}>0$ and $\gamma_{\varphi,s}$ be the constants in \eqref{liminf}, and let  $\varepsilon>0$ and $a\in\widehat{\C}$. 
\begin{itemize}
\item[\textnormal{(a)}] If $\displaystyle\limsup_{r\to\infty}\frac{s(r)}{r}=\infty$ and if $s(r)$
is convex and differentiable, then
	\begin{equation*}
    N(r,a,f(\hat{x}))=N(r,a,f(x) )+O\left(\varphi(r)^{\frac{\rho_\varphi(f)}{\alpha_{\varphi,s}}-\gamma_{\varphi,s}+\varepsilon}\right)+O(\log r),
    \end{equation*}
    $$
    N(r,a,f(\check{x}))=N(r,a,f(x))+O\left(\varphi(r)^{\frac{\rho_\varphi(f)}{\alpha_{\varphi,s}}-\gamma_{\varphi,s}+\varepsilon}\right)+O(\log r).
    $$
\item[\textnormal{(b)}] If $\displaystyle\limsup_{r\to\infty}\frac{s(r)}{r}<\infty$, then
    \begin{equation*}
    N(r,a,f(\hat{x}))=N(r,a,f(x) )+O\left(\varphi(r)^{\rho_\varphi(f)+\varepsilon}\right)+O(\log r),
    \end{equation*}
    $$
    N(r,a,f(\check{x}))=N(r,a,f(x))+O\left(\varphi(r)^{\rho_\varphi(f)+\varepsilon}\right)+O(\log r).
    $$
\end{itemize}
\end{lemma}

Lemma~\ref{N-AW} below is a direct consequence of Lemma \ref{Th5.1-CF} and the definition of the AW-operator $\mathcal{D}_q f$, 
and it reduces to \cite[Theorem~3.3]{CF} when choosing  $\varphi(r)=\log r$ and $s(r)=r^2$.

\begin{lemma}\label{N-AW}
Let $f$ be a non-constant meromorphic function of finite $\varphi$-order $\rho_{\varphi}(f)$. Suppose that $\varphi(r)$ is subadditive. Let $\alpha_{\varphi,s}>0$ and $\gamma_{\varphi,s}$ be the constants in \eqref{liminf}, and let  $\varepsilon>0$.
\begin{itemize}
\item[\textnormal{(a)}] If $\displaystyle\limsup_{r\to\infty}\frac{s(r)}{r}=\infty$ and if $s(r)$
is convex and differentiable, then
	\begin{equation*}
	N\left(r,\mathcal{D}_q f\right)\leq 2N(r,f)+
	O\left(\varphi(r)^{\frac{\rho_\varphi(f)}{\alpha_{\varphi,s}}-\gamma_{\varphi,s}+\varepsilon}\right)+O(\log r).
    \end{equation*}
\item[\textnormal{(b)}] If $\displaystyle\limsup_{r\to\infty}\frac{s(r)}{r}<\infty$, then
   \begin{equation*}
	N\left(r,\mathcal{D}_q f\right)\leq 2N(r,f)+
	O\left(\varphi(r)^{\rho_\varphi(f)+\varepsilon}\right)+O(\log r).
    \end{equation*}
 \end{itemize}
\end{lemma}

The following result reduces to \cite[Theorem~3.4]{CF} when choosing $\varphi(r)=\log r$ and $s(r)=r^2$.

\begin{lemma}\label{T-D_f}
Let $f$ be a non-constant meromorphic function of finite $\varphi$-order $\rho_{\varphi}(f)$. Suppose that $\varphi(r)$ is subadditive. Let $\alpha_{\varphi,s}>0$ and $\gamma_{\varphi,s}$ be the constants in \eqref{liminf}, and let  $\varepsilon\in (0,1)$.
\begin{itemize}
\item[\textnormal{(a)}] If $\displaystyle\limsup_{r\to\infty}\frac{s(r)}{r}=\infty$ and if $s(r)$
is convex and differentiable, then
	\begin{equation*}
	T\left(r,\mathcal{D}_q f\right)\leq 2T(r,f)+
	O\left(\varphi(r)^{\frac{\rho_\varphi(f)}{\alpha_{\varphi,s}}-\gamma_{\varphi,s}+\varepsilon}\right)+O(\log r).
    \end{equation*}
\item[\textnormal{(b)}] If $\displaystyle\limsup_{r\to\infty}\frac{s(r)}{r}<\infty$, then
   \begin{equation*}
	T\left(r,\mathcal{D}_q f\right)\leq 2T(r,f)+
	O\left(\varphi(r)^{\rho_\varphi(f)+\varepsilon}\right)+O(\log r).
    \end{equation*}
 \end{itemize}
\end{lemma}

\begin{proof}
Choose $\varepsilon^*=\frac{\alpha_{\varphi,s}^2\varepsilon^2}{2(\rho_\varphi(f)+\alpha_{\varphi,s}\varepsilon)}\in \left(0,\frac{\alpha_{\varphi,s}}{2}\right)$.
By the definition of the constant $\alpha_{\varphi,s}$ in \eqref{liminf}, it follows that 
	\begin{equation}\label{var-s}
	\varphi(s(r))\leq \varphi(r)^{\frac{1}{\alpha_{\varphi,s}-\varepsilon^*}},\quad r\geq R_0.
	\end{equation}
We replace $\varepsilon$ in Lemma~\ref{m-AW}(a) with 
$\varepsilon'=\frac{\alpha_{\varphi,s}\varepsilon}{2}+\gamma_{\varphi,s}\varepsilon^*=\left(\frac{\alpha_{\varphi,s}}{2}+\frac{\alpha_{\varphi,s}^2\gamma_{\varphi,s}\varepsilon}{2(\rho_\varphi(f)+\alpha_{\varphi,s}\varepsilon)}\right)\varepsilon$, 
which we are allowed to do since $0<\frac{\alpha_{\varphi,s}}{2}\leq \frac{\alpha_{\varphi,s}}{2}+\frac{\alpha_{\varphi,s}^2\gamma_{\varphi,s}\varepsilon}{2(\rho_\varphi(f)+\alpha_{\varphi,s}\varepsilon)}<\alpha_{\varphi,s}\leq 1$.
Consequently, we deduce from \eqref{var-s} that
    \begin{equation}\label{enlarge-term}
    \begin{split}
    \varphi(s(r))^{\rho_\varphi(f)-\alpha_{\varphi,s}\gamma_{\varphi,s}+\varepsilon'}
    &\leq\varphi(r)^\frac{\rho_\varphi(f)-\alpha_{\varphi,s}\gamma_{\varphi,s}
    +\varepsilon'}{\alpha_{\varphi,s}-\varepsilon^*}\\
    &\leq \varphi(r)^{\frac{\rho_\varphi(f)}{\alpha_{\varphi,s}}-\gamma_{\varphi,s}+\varepsilon},\quad r\geq R_0.
    \end{split}
    \end{equation}
Case (a) now follows directly from \eqref{enlarge-term} and Lemmas \ref{m-AW}(a) and \ref{N-AW}(a). Case (b) is more straight forward.
\end{proof}

\begin{remark}\label{2.8-re}
If $\alpha_{\varphi,s}>0$, it is easy to see that 
   \begin{equation*}\label{rho_varphi}
   \rho_\varphi(\mathcal{D}_q f)\leq\max\left\{\rho_{\varphi}(f),\,\frac{\rho_{\varphi}(f)}{\alpha_{\varphi,s}}-\gamma_{\varphi,s}\right\}.
   \end{equation*}
\end{remark}

%%%%%%%%%%%%%%%%%%%%%%%%%%%%%%%%%%%%%%
% SECTION 5
%%%%%%%%%%%%%%%%%%%%%%%%%%%%%%%%%%%%%%

\section{Proofs  of theorems }\label{proofs}
\vskip 3mm
\noindent
{\bf{Proof of Theorem~\ref{AW-th12.4}.}}  
All assertions are true if $\rho_\varphi(f)=\infty$ or if $\alpha_{\varphi,s}=0$, so we may suppose that $\rho_\varphi(f)<\infty$ and $\alpha_{\varphi,s}>0$.  

(a) We begin by proving for every $k\in\N$ that
    \begin{equation}\label{rho_var,k}
    \begin{split}
    \rho_\varphi(\mathcal{D}_q^{k} f)&\leq {\max}\left\{\rho_\varphi(f),\, \max_{1\leq l\leq k}\left\{\frac{\rho_\varphi(f)}{\alpha_{\varphi,s}^{l}}-\gamma_{\varphi,s}\sum_{j=0}^{l-1}\frac{1}{\alpha_{\varphi,s}^j}\right\}\right\}=:\rho_{\varphi,k}.
    \end{split}
    \end{equation}
The case $k=1$ is obvious by  Remark~\ref{2.8-re}. We suppose that \eqref{rho_var,k} holds for $k$, and we aim to prove 
\eqref{rho_var,k} for $k+1$. Applying Remark~\ref{2.8-re} to the meromorphic function $\mathcal{D}_q^{k} f$ yields
       \begin{equation*}
      \begin{split}
      \rho_\varphi(\mathcal{D}_q^{k+1} f)&=\rho_\varphi(\mathcal{D}_q(\mathcal{D}_q^{k} f))
      \leq\max\left\{\rho_\varphi(\mathcal{D}_q^{k} f),\,
      \frac{\rho_\varphi(\mathcal{D}_q^{k} f)}{\alpha_{\varphi,s}}-\gamma_{\varphi,s}\right\}\\
      &\leq\max\left\{\rho_\varphi(f),\,\max_{1\leq l\leq k+1}\left\{\frac{\rho_\varphi(f)}{\alpha_{\varphi,s}^{l}}-\gamma_{\varphi,s}\sum_{j=0}^{l-1}\frac{1}{\alpha_{\varphi,s}^j}\right\}\right\}=\rho_{\varphi,k+1}.
      \end{split}
      \end{equation*}
The assertion  \eqref{rho_var,k} is now proved. Moreover, it is easy to see that 
$\rho_\varphi(f)\leq \rho_{\varphi,k}\leq \rho_{\varphi,k+1}$ for $k\in\N$.
       
Suppose first that the coefficients $a_0(x),\ldots,a_n(x)$ are entire.
We divide \eqref{AW-q-diff} by $f(x)$ and make use of \eqref{enlarge-term}, \eqref{rho_var,k} and Lemma \ref{m-AW}(a) to obtain
    \begin{equation}\label{m-a_0-Dq}
    \begin{split}
	m(r,a_0)&\leq \max_{1\leq j\leq n}\{m(r,a_j)\}+{\sum_{1\leq j\leq n}} m\left(r,\frac{\mathcal{D}_q^jf}{f}\right)\\
	&\lesssim \max_{1\leq j\leq n}\{m(r,a_j)\}+{\max_{1\leq j\leq n}}\left\{m\left(r,\frac{\mathcal{D}_q^{j} f}{\mathcal{D}_q^{j-1}  f}\right)\right\}\\
		&\lesssim  \varphi(r)^{\rho_\varphi(a_0)-\veps} 
	+ \varphi(r)^{\frac{\rho_{\varphi,n-1}}{\alpha_{\varphi,s}}-\gamma_{\varphi,s}+{\varepsilon}},\quad r\geq R_0.  
	\end{split}
    \end{equation}
 Since there exists a sequence $\{r_n\}$ of positive real numbers tending to infinity  such that $m(r_n,a_0)\geq \varphi(r_n)^{\rho_\varphi(a_0)-\frac{\varepsilon}{2}}$, we have
    $$
    \rho_\varphi(a_0)-\frac{\varepsilon}{2}\leq \frac{\rho_{\varphi,n-1}}{\alpha_{\varphi,s}}-\gamma_{\varphi,s}+{\varepsilon},
    $$
where we may let $\varepsilon\to 0^+$. This gives us  
    \begin{eqnarray*}
     \rho_\varphi(a_0) &\leq& \max\left\{\frac{\rho_\varphi(f)}{\alpha_{\varphi,s}}-\gamma_{\varphi,s},\, 
     \max_{1\leq l\leq n-1}\left\{\frac{\rho_\varphi(f)}{\alpha_{\varphi,s}^{l+1}}-\gamma_{\varphi,s}\sum_{j=0}^{l}\frac{1}{\alpha_{\varphi,s}^j}\right\}\right\}\\
     &=&\max_{1\leq l\leq n}\left\{\frac{\rho_\varphi(f)}{\alpha_{\varphi,s}^{l}}-\gamma_{\varphi,s}\sum_{j=0}^{l-1}\frac{1}{\alpha_{\varphi,s}^j}\right\},
    \end{eqnarray*}
and so
    \begin{eqnarray*}
     \alpha_{\varphi,s}^n \rho_\varphi(a_0)
     &\leq&\max_{1\leq l\leq n}\left\{\alpha_{\varphi,s}^{n-l}\rho_\varphi(f)
     -\gamma_{\varphi,s}\sum_{j=0}^{l-1}\alpha_{\varphi,s}^{n-j} \right\}
     \leq \rho_\varphi(f)-\alpha_{\varphi,s}^n \gamma_{\varphi,s}.
    \end{eqnarray*}
    Then the assertion \eqref{a-2-2.1} follows.

Suppose then that some of the coefficients $a_0(x),\ldots,a_n(x)$  have poles. We divide \eqref{AW-q-diff} by $f(x)$ and make use of   \eqref{rho_var,k} and Lemma \ref{T-D_f}(a) to obtain
   \begin{equation*}\label{N--a-0}
    \begin{split}
	N(r,a_0)&\lesssim \max_{1\leq j\leq n}\{T(r,a_j)\}+\sum_{j=0}^n
T(r,\mathcal{D}_q^j f) 	\\
    &\lesssim \max_{1\leq j\leq n}\{T(r,a_j)\}+T(r,f)+
	\varphi(r)^{\frac{\rho_{\varphi,n-1}}{\alpha_{\varphi,s}}-\gamma_{\varphi,s}+\varepsilon}+\log r\\
	&\lesssim \varphi(r)^{\rho_\varphi(a_0)-\veps}+
	\varphi(r)^{\rho_\varphi(f)+\varepsilon}+\varphi(r)^{\frac{\rho_{\varphi,n-1}}{\alpha_{\varphi,s}}-\gamma_{\varphi,s}+\varepsilon}+\log r,\quad r\geq R_0.
	\end{split}
    \end{equation*}
Combining this with \eqref{m-a_0-Dq} and noting the fact that $f$ is non-constant, we obtain
    \begin{equation*}\label{rho-(a_0)-N}
    \rho_\varphi(a_0)\leq \max\left\{\rho_\varphi(f), \, \max_{1\leq l\leq n}\left\{\frac{\rho_\varphi(f)}{\alpha_{\varphi,s}^{l}}-\gamma_{\varphi,s}\sum_{j=0}^{l-1}\frac{1}{\alpha_{\varphi,s}^j}\right\}\right\}=\rho_{\varphi,n},
    \end{equation*}
and thus, similarly as above,
    \begin{equation*}
    \begin{split}
    \alpha_{\varphi,s}^n \rho_\varphi(a_0)&\leq \max\left\{ \alpha_{\varphi,s}^n \rho_\varphi(f), \,  
    \rho_\varphi(f)-\alpha_{\varphi,s}^n \gamma_{\varphi,s}\right\}\leq \rho_\varphi(f).
    \end{split}
    \end{equation*}
Hence the assertion \eqref{a-1-2.1} follows.

(b)  Similarly as in Case (a) above, we make use of \eqref{rho_var,k}  and Lemmas \ref{m-AW}(b) and \ref{T-D_f}(b) to obtain
     \begin{equation*}
    \begin{split}
T(r,a_0)&=m(r,a_0)+N(r,a_0)\\& \lesssim \max_{1\leq j\leq n}\{T(r,a_j)\}+{\sum_{1\leq j\leq n}} m\left(r,\frac{\mathcal{D}_q^jf}{f}\right)+\sum_{j=0}^n
T(r,\mathcal{D}_q^j f) \\
	&\lesssim  \varphi(r)^{\rho_\varphi(a_0)-\veps} 
	+\varphi(r)^{\rho_{\varphi,n-1}+\varepsilon}+\log r,\quad r\geq R_0. 
	\end{split}
    \end{equation*}
This together with the fact that $f$ is non-constant, we deduce  $\rho_{\varphi}(a_0)\leq \rho_{\varphi,n-1}$, and so        
$\alpha_{\varphi,s}^{n-1} \rho_\varphi(a_0)\leq \rho_\varphi(f).$
This completes the proof. \hfill$\Box$

\vskip 3mm
\noindent
{\bf{Proof of Theorem~\ref{AW-th12.4-non}.}}  
Choose $s(r)$ satisfying the assumptions of Theorem~\ref{AW-th12.4}(b).   We divide \eqref{AW-q-diff-non} by $f(x)$ and make use of   \eqref{rho_var,k} and Lemmas \ref{m-AW}(b) and \ref{T-D_f}(b) to obtain
    \begin{equation*}
    \begin{split}
T(r,a_0) & \lesssim \max_{1\leq j\leq n+1}\{T(r,a_j)\}+{\sum_{1\leq j\leq n}} m\left(r,\frac{\mathcal{D}_q^jf}{f}\right)+m\left(r,\frac{1}{f}\right)+\sum_{j=0}^n
T(r,\mathcal{D}_q^j f)\\
		&\lesssim  \varphi(r)^{\rho_\varphi(a_0)-\veps} 
	+\varphi(r)^{\rho_{\varphi,n-1}+\varepsilon}+\log r,\quad r\geq R_0.
	\end{split}
    \end{equation*}
Similarly as in the proof of Theorem~\ref{AW-th12.4}(b), the assertion follows. \hfill$\Box$

\section*{Acknowledgements}

The first author would like to thank the support of the China Scholarship Council (No.~201806330120).
The third author was supported by National Natural Science Foundation of China (No.~11771090).
The fourth author was supported by the National Natural Science Foundation of China (No.~11971288 and No.~11771090) and Shantou University SRFT (NTF18029).

\end{document}